\documentclass[11pt]{amsart}
\usepackage{amsfonts,amscd,amssymb,amsmath,amsthm,mathrsfs,xcolor,lscape,amsmath,amssymb,latexsym, booktabs, rotating, graphicx}
\usepackage{enumerate,enumitem,pifont}
\usepackage[utf8]{inputenc}
\usepackage[colorlinks = true, citecolor = blue,]{hyperref}
\usepackage[english]{babel}
\newtheorem{thm}{Theorem}
\newtheorem{lem}[thm]{Lemma}
\newtheorem{defn}{Definition}

\newtheorem{que}[thm]{Question}
\theoremstyle{remark}
\theoremstyle{definition}
\newtheorem{rmk}{Remark}
\numberwithin{equation}{section}
\numberwithin{rmk}{section}
\numberwithin{thm}{section}
\numberwithin{defn}{section}

\newcommand{\C}{\mathbb{C}}
\newcommand{\N}{\mathbb{N}}
\newcommand{\R}{\mathbb{R}}

\newcommand{\Sp}{\mathrm{Sp}}

\newcommand{\GL}{\mathrm{GL}}

\newcommand{\Q}{\mathbb{Q}}
\newcommand{\Z}{\mathbb{Z}}
\newcommand{\e}{\epsilon}

\newcommand{\Om}{\Omega}

\newcommand{\spn}{\text{span}}

\title{Some Infinite Families of Arithmetic Symplectic Hypergeometric Groups}
\author{Lal Bahadur Sahu}

\address{Department of Mathematics, Indian Institue of Technology Bombay, Powai, Mumbai. 400076. INDIA.}
\email{215090034@iitb.ac.in}
\subjclass[2010]{Primary: 22E40;  Secondary: 32S40;  33C80}  
\keywords{Hypergeometric group, Monodromy representation, Symplectic group}

\begin{document}

\begin{abstract}
Using the idea of Venkataramana's construction of infinite families of arithmetic orthogonal hypergeometric groups, we construct some of the very first examples of infinite families of arithmetic symplectic hypergeometric groups that do NOT satisfy the arithmeticity criterion of Singh and Venkataramana. 

For example, we show that the hypergeometric groups associated to the pairs of polynomials $(x-1)^4P_m(x^{9})$ and $(x^4+x^3+2x^2+x+1)Q_m(x^{9})$, where $P_m$ and $Q_m$ are integral polynomials of degree $2m$ such that the corresponding pair determines a Zariski dense symplectic hypergeometric group, are arithmetic in $\Sp(18m+4)$ for any integer $m\in\N$.

\end{abstract}
\maketitle

\section{Introduction}
The primary finding of this article relates to monodromy groups of hypergeometric differential equations of type ${}_nF_{n-1}$ in the symplectic case. In \cite{B-H89}, these equations are thoroughly examined. Let $\alpha,\beta\in \C^n$ and consider the hypergeometric differential equation
\begin{equation}\label{HDE}
	D(\alpha;\beta)u = 0
\end{equation} 
on the punctured Riemann sphere $\mathbb P^1(\C)\setminus \{0,1,\infty\}$, where 
\[D(\alpha;\beta) := \left(z\frac{d}{dz} +\beta_1-1\right)\cdots\left(z\frac{d}{dz} +\beta_n-1\right)-z\left(z\frac{d}{dz} +\alpha_1\right)\cdots\left(z\frac{d}{dz} +\alpha_n\right).\]
The Equation (\ref{HDE}) is regular outside of the set $\{0,1,\infty\}$ in $\mathbb P^1(\C)$. Let $z_0\in \mathbb P^1(\C)\setminus \{0,1,\infty\}$ and let $h_0,h_\infty, h_1$ be the path-homotopy classes of loops, based at $z_0$, around $0,\infty, 1$, respectively. Then the fundamental group of $\mathbb P^1(\C)\setminus\{0,1,\infty\}$ is generated by $h_0,h_\infty,h_1$ modulo the relation that $h_0h_1h_\infty = 1$, and it acts on the local solution space of the hypergeometric differential equation (\ref{HDE}) by analytic continuation of solutions along the loops representing $h_0$, $h_\infty$ and $h_1$.This action is called a \textit{monodromy representation} and the image of this representation sitting inside $\GL_n(\C)$ is called \textit{monodromy group}, and a \textit{hypergeometric group} in our setting. 

Levelt, in his thesis \cite{L61}, showed the existence of a basis of the solution space of (\ref{HDE}), under the condition that the pair $(f,g)$ is coprime, with respect to which, the associated hypergeometric group is the subgroup of $\GL_n(\C)$ generated by the companion matrices $A$ and $B$ of the polynomials
$$ f(x) = \prod_{j=1}^{n}(x-a_j), \quad g(x) = \prod_{j=1}^{n} (x - b_j)$$
respectively, where $a_j = e^{2\pi\iota \alpha_j}$ and $b_j =e^{2\pi\iota\beta_j}$, and the monodromy representation is defined by 
$$h_\infty \mapsto A, \quad h_0\mapsto B^{-1}, \quad h_1\mapsto A^{-1}B.$$

Under the condition of Levelt, the hypergeometric group is generated by the companion matrices of the polynomials $f,g$ and we denote it as $\Gamma(f,g)\subseteq \GL_n(\C)$. In this article, we consider $f,g$ to be integral polynomials with $f(0)=g(0)= 1$. The Zariski closure ${\mathbf G}$ of the aforementioned hypergeometric group $\Gamma(f,g)$ inside $\GL_n(\C)$ has been thoroughly examined by Beukers and Heckman \cite{B-H89}. We now give a brief explanation of their findings under the premise that $(f,g)$ are coprime, \textit{self-reciprocal}; i.e., $\{a_1,\ldots,a_n\}$ and $\{b_1, \ldots, b_n\}$, roots of $f$ and $g$ respectively, are invariant under the inverse map and $(f,g)$ form a \textit{primitive pair} in the sense of \cite{B-H89}; i.e., there does not exist an integer $k\ge 2$ such that $f(x)$ and $g(x)$ both can be written as polynomials in $x^k$ (see \cite[Theorem 5.3]{B-H89}). 

In \cite[Table 8.3]{B-H89}, Beukers and Heckman have completely classified the finite hypergeometric groups. These groups correspond to the parameters (which are the roots of the associated polynomials) that interlace on the unit circle (\cite[Corollary 4.7]{B-H89}).  For an infinite hypergeometric group $\Gamma(f,g)$, which is not a scalar shift of a finite hypergeometric group, they characterize the Zariski closure ${\mathbf G}$ in two classes based on $\frac{f(0)}{g(0)}$ being $1$ or $-1$. If $\frac{f(0)}{g(0)} = -1$, then ${\mathbf{G}} = \mathrm O_Q$, where $Q$ is a non-degenerate quadratic form on $\Q^n$ preserved by $\Gamma(f,g)$, whose existence is guaranteed by the self-reciprocity of $f$ and $g$ both. In this article, we study the other case, i.e.,  $\frac{f(0)}{g(0)} =1$, and in this case ${\mathbf{G}}=\Sp_\Omega$, where $\Omega$ is a non-degenerate symplectic form on $\Q^n$ and preserved by $\Gamma(f,g)$.

In this article we focus on the hypergeometric groups $\Gamma(f,g)$, where the defining polynomials $(f,g)$ have integer coefficients and $f(0)=g(0)=1$. It follows from the definition of $\Gamma(f,g)$ that  $\Gamma(f,g)\subseteq {\mathbf{G}}(\Z)$ in this case. We now note the following definition.

\begin{defn}
	We call a hypergeometric group $\Gamma(f,g)\subseteq {\mathbf{G}}(\Z)$ arithmetic if it is of finite index in ${\mathbf{G}}(\Z)$, and thin if it has infinite index in ${\mathbf{G}}(\Z)$, where ${\mathbf{G}}$ is the Zariski closure of $\Gamma(f,g)$ inside $\GL_n(\C)$.
\end{defn}

Considerable attention has recently been directed toward Sarnak’s question \cite{PS14}, which seeks to classify the polynomial pairs $(f, g)$ whose corresponding hypergeometric groups $\Gamma(f, g)$ are arithmetic versus those that are thin. Recent progress in classifying thin hypergeometric groups is detailed in \cite{F-M-S14, B-T14, B-D-N25, B-N24, F-F24, S-S24},  while corresponding advancements for arithmetic hypergeometric groups can be found in \cite{ S-V14, TNV14, B-S15, SS15, SS17, TNV17, B-D-S-S21, B-S-S23, B-D-N26}. 

 In this article we provide at least eight infinite families of arithmetic symplectic hypergeometric groups corresponding to pair of polynomials $(f,g)$ such that the leading coefficient of $f-g$ has absolute value $\ge 3$. Because of this condition, the arithmeticity of the hypergeometric groups associated to these pairs cannot be deduced using the celebrated arithmeticity criterion of Singh and Venkataramana \cite[Theorem 1.1]{S-V14}. Moreover, in some cases it is impossible to prove the arithmeticity even by applying the extended arithmeticity criterion \cite[Remark 5.1]{S-V14} used by Bajpai, Dona,  S. Singh and S.V. Singh \cite[Proposition 1]{B-D-S-S21}. 

Finally, Table \ref{Table 1} lists polynomial pairs $(f_0,g_0)$ whose arithmeticity has been established in the literature, alongside three specific transvections. The subgroup of $\Gamma(f_0,g_0)$ generated by these three transvections contains unipotent elements corresponding to the highest and second highest roots of $\Sp(4)$. Throughout this paper, the companion matrices of $f_0$ and $g_0$ are denoted by $A_0$ and $B_0$, respectively, with $C_0 = A_0^{-1}B_0$.

\begin{table}[ht]
	\centering
	\renewcommand{\arraystretch}{1.4} 
	\caption{Arithmetic symplectic pairs of polynomials of degree four that do not satisfy the criterion of Singh and Venkataramana (but their arithmeticity is proved using the three transvections listed next to them).}
	\label{Table 1}
	\resizebox{\textwidth}{!}{%
	\begin{tabular}{ccccccc}
		\toprule
		No. & $f_0$ & $g_0$  & Transvections & Source \\ 
		\midrule
		1 & $(x-1)^4$ & $(x^2+x+1)(x^2+1)$ & $B_0^{3}C_0B_0^{-3}, C_0, B_0^{-3}C_0B_0^{3}$ & \cite[2.1]{SS15} \\
		2 & $(x-1)^4 $ & $(x^2+x+1)(x^2-x+1)$  & $B_0^{3}C_0B_0^{-3}, C_0, B_0^{-1}C_0B_0$ & \cite[2.3]{SS15} \\
		3 & $(x-1)^4$& $(x^2+1)^2$  & $B_0^{3}C_0B_0^{-3}, C_0, B_0^{-5}C_0B_0^{5 }$ & \cite[2.4]{SS15} \\
		4 & $(x^2+x+1)^2$ & $(x-1)^2(x^2+1)$  & $A_0^{4}C_0A_0^{-4}, C_0, A_0^{-4}C_0A_0^{4 }$ & \cite[3.2]{SS17}\\
		5 & $(x^2+x+1)^2 $&$ (x-1)^2(x^2-x+1)$  & $A_0^{4}C_0A_0^{-4}, C_0, A_0^{-4}C_0A_0^{4 }$ & \cite[3.3]{SS17}\\
		6 & $(x^2+x+1)^2$ & $(x^2+1)(x^2-x+1)$  & $B_0^{3}C_0B_0^{-3}, C_0, B_0^{-3}C_0B_0^{3}$ & \cite[3.5]{SS17} \\
		7 & $(x^2+x+1)^2$ &$ (x^4-x^3+x^2-x+1)$  & $A_0^{4}C_0A_0^{-4}, C_0, A_0^{-4}C_0A_0^{4 }$ & \cite[3.6]{SS17}\\
		8 & $(x+1)^2(x^2+x+1)$ &$(x^2+1)(x^2-x+1)$  & $B_0^{3}C_0B_0^{-3}, C_0, B_0^{-3}C_0B_0^{3}$ & \cite[3.7]{SS17} \\
		\bottomrule
	\end{tabular}%
	}
\end{table}

We now state our main theorem.
\begin{thm}\label{main thm}
Let  $m \in\N$, and let $(f_0, g_0)$ be a pair of polynomials listed in Table \ref{Table 1}. Suppose $P_m$ and $Q_m$ are monic, self-reciprocal polynomials of degree $2m$ with integer coefficients, satisfying $P_m(0) = Q_m(0) = 1$. Let $f(x) = f_0(x)P_m(x^{9})$ and $g(x) = g_0(x)Q_m(x^{9})$. If $f$ and $g$ are coprime, then the group $\Gamma(f,g)$ is an arithmetic subgroup of the  symplectic group $\Sp(18m+4)$. 
\end{thm}

We note the following remarks.

\begin{rmk}
    Note that the self-reciprocal and coprime condition, along with the symplectic condition $P_m(0)=Q_m(0)=1$, on the pair of the polynomials $P_m$ and $Q_m$ forces the degrees of these polynomials to be an even integer.
\end{rmk}

\begin{rmk}
As far as we know, the infinite families constructed in Theorem \ref{main thm} provide very first examples of an infinite family of arithmetic symplectic hypergeometric groups for which the generating polynomials do NOT satisfy the arithmeticity criterion of Singh and Venkataramana. Another infinite family of arithmetic symplectic hypergeometric groups is constructed in \cite[Theorem 25]{TNV14} but the hypergeometric groups of this family do satisfy the arithmeticity criterion of Singh and Venkataramana.
\end{rmk}

\begin{rmk}
	In Theorem \ref{main thm} we are not assuming $P_m,Q_m$ to be a product of cyclotomic polynomials (but they must be self-reciprocal for obvious reasons). The following example of the pair  $P_m, Q_m$ works uniformly for all the pairs listed in Table \ref{Table 1}:
    \[P_m(x)= (x^2+3x+1)^{m}, \quad Q_m(x) = (x^2+4x+1)^{m}.\]
     Observe that the above polynomials are not product of cyclotomic polynomials. We provide another example of the pair $P_m, Q_m$, which are product of cyclotomic polynomials:
    \[P_m(x)= (x^2+x+1)^{m}, \quad Q_m(x) = (x^2+1)^{m}.\]
\end{rmk}

\begin{rmk}\label{primitive pair}
	Since the coefficient of $x$ in $f_0$ is nonzero in all the eight cases of Table \ref{Table 1}, the pair $(f,g)$ of Theorem \ref{main thm} does form a primitive pair.
\end{rmk}
%

\begin{rmk}
	It is remarked in Section \ref{section 5} that the method of the proof of Theorem \ref{main thm} can easily be used to provide more infinite families of arithmetic symplectic hypergeometric groups by taking the polynomial pair $(f_0,g_0)$ of degree $6$ with slight change in the exponent of $x$ appearing as $x^{9}$ in $P_m(x^{9})$ and $Q_m(x^{9})$ (one may need to consider $P_m(x^{k})$ and $Q_m(x^{k})$ for some other $k\in\N$).
\end{rmk}

\begin{rmk}
	To prove arithmeticity, the unipotent radical technique is used by Venkataramana in \cite[Theorem 7]{TNV17} due to the $\Q$-rank not always being equal to $\R$-rank in the orthogonal case. Since it is not so in the symplectic case, the proof becomes a bit simpler.
\end{rmk}

\subsection{Strategy of the proof of Theorem \ref{main thm}} 

The main idea of the proof is to use the arithmeticity criterion of Venkataramana \cite[Theorem 3.5]{TNV87} by showing that some reflection subgroup of $\Gamma(f,g)$ intersects the highest and second highest root groups of the corresponding symplectic group $\Sp(18m+4)$ non-trivially. To prove it, we first show that $\Gamma(f,g)$ contains a `copy' of the subgroup of $\Gamma(f_0,g_0)$ generated by the three transvections listed in Table \ref{Table 1}. It is already known (see the ``Transvections" and ``Source" columns of Table \ref{Table 1}) that the subgroup generated by the listed transvections in $\Gamma(f_0,g_0)$ contains unipotent elements corresponding to the highest and second highest roots of $\Sp(4)$. We choose a basis of $\Q^{18m+4}$, with respect to which, the copies of the unipotent elements of $\Gamma(f_0,g_0)$ in $\Gamma(f,g)$ correspond to the highest and second highest roots of $\Sp(18m+4)$. This idea of the proof is originated from the paper \cite{TNV17} of Venkataramana where he constructs infinite families of arithmetic orthogonal hypergeometric groups. We translate his idea to the symplectic case.

\section{Proof of the Theorem \ref{main thm}}

\subsection{Discussion of action and form}
Consider the vector space $V=\Q[x]/\langle g(x)\rangle$, where $g(x)$ is the polynomial given in Theorem \ref{main thm}. It follows that the companion matrix $B$ of the polynomial $g(x)$ is the matrix, with respect to the basis $\{1=x^0,x,\ldots, x^{n-1}\}$, of the operator acting on $V$ via multiplication by $x$. It also follows that the companion matrix $A$ of the polynomial $f(x)$ given in Theorem \ref{main thm} is the matrix, with respect to the same basis, of the operator acting on $V$ via multiplication by $x$ but sending the basis vector $x^{n-1}$ to $x^n-f(x)$. Let $C = A^{-1}B$. Clearly, $C$ fixes the basis vectors $1,x,\ldots, x^{n-2}$ and sends the basis vector $x^{n-1}$ to the vector $x^{n-1}+\frac{f(x)-g(x)}{x}$ of $V$. Let us denote the vector $(C-I)(x^{n-1})=\frac{f(x)-g(x)}{x}$ by $v$.

Observe that the polynomials $(f, g)$, appearing in Theorem \ref{main thm}, are self-reciprocal polynomials and form a primitive pair. Since the polynomials $(f, g)$ are also co-prime, it follows from \cite{B-H89} that there exists a unique (up to scalar multiplications) non-degenerate symplectic form $\Om$ on $V$ that is preserved by $\Gamma(f,g)$. 

We now note the following lemma from \cite[Lemma 4.1]{S-V14} (cf. \cite[Remark 3]{TNV17}). For the sake of completeness, we also note a proof of the lemma.

\begin{lem}\label{v is cyclic}
	The vector $v$ is orthogonal to the first $n-1$ basis vectors, $1, x, \ldots, x^{n-2}$, with respect to the symplectic form $\Om$. Moreover, $v$ is cyclic for the action of $B$ (and also for the action of $A$) on $V$.
\end{lem}

\begin{proof}
	Using the $\Gamma(f,g)$-invariance of $\Om$, we have for all $i\in \{0,1,\ldots, n-2\}$,
	\[\Omega(x^i,x^{n-1}) = \Omega(Cx^i,Cx^{n-1}) = \Omega(x^{i},x^{n-1}+v) = \Omega(x^i,x^{n-1})+\Omega(x^i,v).\]
	Canceling $\Omega(x^i,x^{n-1})$ yields $\Omega(x^i,v) = 0$ for $i =0,1,\ldots, n-2$. 
	
	We show here that the vector $Bv$ is a cyclic vector for the action of $B$ on $V$, and this implies that $v$ is also a cyclic vector for the action of $B$ on $V$.  Let, if possible, $Bv$ is not a cyclic vector for the action of $B$, that is, the subspace $U = \spn\{Bv,B^2v,\ldots\}$ is a proper subspace of $V$. Observe that for any $w\in \Q^n$, $Aw = Bw - \alpha_w Bv$, where $\alpha_w$ is the coefficient of $x^{n-1}$ in $w$. This implies that $U$ is an $A$-invariant subspace of $V$. By definition, $U$ is also $B$-invariant. It shows that the action of $\Gamma(f,g)=\langle A,B\rangle$ on $V$ is not irreducible and this is a contradiction to \cite[Proposition 3.3]{B-H89}. Hence, $Bv$ is cyclic, and therefore $v$ is also cyclic for the action of $B$ on $V$.      
\end{proof}

Since $\Om$ is non-degenerate, it follows from Lemma \ref{v is cyclic} that $\Om(v,x^{n-1})\neq 0$. We normalize $\Om$ in such a way that $\Om(v,x^{n-1}) = 1$. Therefore, for any vector $w = q_0+ q_1x+\cdots+ q_{n-1}x^{n-1}\in V$ with $q_i\in\Q$, we have $\Om(v,w) = q_{n-1}.$ 

Since $(C-I)(x^{n-1})=v$, it follows that $Bv = Av = f(x)-g(x)$, a polynomial of degree at most $n-1$. It is important to note that $Bv = Av$ holds in symplectic case but it is not so in orthogonal case because the coefficient of $x^{n-1}$ in $v$ is $-2$ there. Having the condition $Bv = Av$ is very crucial for the hypothesis of Theorem \ref{main thm} to get hold for all the examples of Table \ref{Table 1}. 

\begin{rmk}\label{form calculation}
	Since $v$ is cyclic for the action of $B$, $\{B^{-1}v, v, Bv, \ldots, B^{n-2}v\}$ forms a basis for the vector space $V$ (this basis is used for Example $2$ of Table \ref{Table 1}). To determine the symplectic form, we need to calculate $\Om(B^kv,B^lv)$ for $-1\le k,l\le n-2$. But the $\Gamma(f,g)$-invariance of the symplectic form $\Om$ implies that it is sufficient to calculate $\Om(v,B^kv)$ for $0\le k\le n-1$. As discussed in the above paragraph, $\Om(v,B^kv)$ is the coefficient of $x^{n-1}$ in the linear combination of the vector $B^kv$, when written as the linear combination of the basis vectors $1,x, \ldots, x^{n-1}$. Thus, by the observation made, $\Om(v,Bv)$ is just the coefficient of $x^{n-1}$ in $f-g$. Since $B$ acts on $V$ as multiplication by $x$, we only have to look at the coefficient of $x^{n-1}$ in $x^{k-1}(f-g)$ to find $\Om(v,B^kv)$ for $k\ge 2.$    
\end{rmk}

\subsection{Transfer of particular unipotent elements}
We now examine the case of degree four polynomial pairs $(f_0,g_0)$ listed in the Table \ref{Table 1}. To understand it, we take Example $2$ from the Table \ref{Table 1}. In this example $f_0(x) = (x-1)^4$ and $g_0(x) = (x^2+x+1)(x^2-x+1)$. Clearly, $f_0,g_0$ form a primitive pair. The companion matrices $A_0$ and $B_0$ associated to $f_0$ and $g_0$, respectively, give rise to the associated hypergeometric group, denoted as $\Gamma(f_0, g_0)$ in $\Sp(4)$. This subgroup preserves a non-degenerate (normalized) symplectic form $\Omega_0$ on the vector space $V_0 = \Q[x]/\langle g_0(x)\rangle$. As it is mentioned above, the cyclic vector $v_0\in V_0$, in this particular case, is $v_0 = -4x^2+5x-4$. The elements $w_0$ of $V_0$ are expressed as polynomials of degree less than or equal to $3$, that is, $w_0 = q_{03} x^3 +q_{02} x^2 +q_{01}x +q_{00}$, where $q_{0i}\in \Q$. It is clear that the set $\{B_0^{-1}v_0, v_0, B_0^2 v_0\}$, being a subset of a basis generating the vector space $V_0$ generated by the cyclic vector $B_0^{-1}v_0$ under the action of $B_0$, is linearly independent. An easy calculation shows that $B_0^3v_0 = -5x^2+4x-5$, which can not be written as a linear combination of the vectors $B_0^{-1}v_0, v_0, B_0^2 v_0$. Since $V_0$ is a $4$-dimensional vector space, it is spanned by the four linearly independent vectors $B_0^{-1}v_0, v_0, B_0^2 v_0$ and $B_0^3v_0$.

With the foregoing hypothesis of Theorem \ref{main thm}, we get that $f(x) = f_0(x)P_m(x^{9})$ and $g(x) = g_0(x)Q_m(x^{9})$ are monic polynomials of degree $n = 2(9m+2)$, where $2m$ is the degree of the polynomials $P_m(x), Q_m(x)\in \Z[x]$, with $P_m(0) = Q_m(0) = 1$ and $(f,g)$ is coprime. These conditions forces that $f(0) = g(0)= 1$. Along with these conditions, the pair $(f,g)$ also forms a primitive pair (cf. Remark \ref{primitive pair}). 

Let $W$ be the subspace, spanned by the vectors $B^{-1}v, v, B^2v, B^3v$,  of the vector space $V$, where $v$ is the same vector introduced at the beginning of this section. It is an easy observation that $W$ is a $4$-dimensional vector subspace as the degree of the polynomials $f,g$ is greater or equal to $22$ and it follows from  Lemma \ref{v is cyclic} that the set $\{B^{-1}v, v, B^2v, B^3v\}$ is linearly independent in $V$. It is also easy to see that there is a vector space isomorphism 
$$\phi: V_0\to W\subset V,  \mbox{ defined by sending } B^l_0v_0\mapsto B^lv, \text{ for }l=-1,0,2,3.$$

We now note down Lemma \ref{lc in x^3 and x^{n-1}}, which is an easy generalization of \cite[Lemma 4]{TNV17}, to make it work for the examples of Table \ref{Table 1} whose transvections are conjugates of $C_0=A_0^{-1}B_0$ by powers of $B_0$. The Lemma \ref{lc in x^3 and x^{n-1}} is used in the proof of Lemma \ref{symplectic iso}, which is again a replica of the idea laid out in \cite[Lemma 5]{TNV17}.

\begin{lem}\label{lc in x^3 and x^{n-1}}
	Let $p$ be a positive integer. If $j\le 3,$ then the coefficient of $x^{n-1}$ in the remainder of $x^j(f_0(x)P_m(x^p)-g_0(x)Q_m(x^p))$ upon division by $g_0(x)Q_m(x^p)$ is the same as the coefficient of $x^3$ in the remainder of $x^j(f_0(x)- g_0(x))$ upon division by $g_0(x)$ for $p\ge 5$.
\end{lem}

\begin{proof}
	For $x^j(f_0(x)P_m(x^p)-g_0(x)Q_m(x^p))$, there exist $q(x),r(x)\in \Z[x]$ such that 
    \begin{equation}\label{coefficient}
        x^j(f_0(x)P_m(x^p)-g_0(x)Q_m(x^p))= q(x)g_0(x)Q_m(x^p)+r(x)
    \end{equation}
    where $\deg r(x)< \deg g_0(x)Q_m(x^p)$; and for $x^j(f_0(x)-g_0(x))$, there exist $q_0(x),r_0(x)\in \Z[x]$ such that
\begin{equation}
    x^j(f_0(x)-g_0(x)) = q_0(x)g_0(x)+r_0(x)
\end{equation}
    where $\deg r_0(x)< \deg g_0(x)$. Also, Observe that
\begin{equation}\label{coefficient of x^{n-1}}
    x^j(f_0(x)P_m(x^p)-g_0(x)Q_m(x^p)) = x^j(f_0(x)-g_0(x))Q_m(x^p)+x^jf_0(x)(P_m(x^p)-Q_m(x^p))
\end{equation}
    and
    \begin{equation}\label{coefficient of x^3}
        x^j(f_0(x)-g_0(x))Q_m(x^p)= q_0(x)g_0(x)Q_m(x^p)+r_0(x)Q_m(x^p).
    \end{equation}
Since the degree of $x^jf_0(x)(P_m(x^p)- Q_m(x^p))$ is at most $p(2m-1)+7$ and that of $r_0(x)Q_m(x^p)$ is at most $2pm+3$, it follows from Equations (\ref{coefficient of x^{n-1}}) and (\ref{coefficient of x^3}) that the remainder $r(x)$ appearing in Equation (\ref{coefficient}) is given by the equation
\begin{equation}\label{remainder}
    r(x)=r_0(x)Q_m(x^p)+x^jf_0(x)(P_m(x^p)-Q_m(x^p))
\end{equation}
if $p\ge 5$ (in this case the degree of the polynomial $g_0(x)Q_m(x^p)$ is $2mp+4$, which is larger than the maximum of $p(2m-1)+7$ and  $2pm+3$).

	 Finally, it follows from Equation (\ref{remainder}) that if $p\ge 5$, then $2pm-p+7<2pm+3$ and hence the coefficient of $x^{n-1} =x^{2pm+4-1}=x^{2mp+3}$ in $r(x)$ is equal to the coefficient of $x^{2mp+3}$ in $r_0(x)Q_m(x^p)$ but the coefficient of $x^{2mp+3}$ in $r_0(x)Q_m(x^p)$ is equal to the coefficient of $x^3$ in $r_0(x)$ (since the polynomial $Q_m(x^p)$ is a monic polynomial of degree $2mp$) and the lemma follows.
\end{proof}

Now, for each $l\in \Z$, consider the linear operator defined on $V$ as 
$$ w\mapsto w - \Om(w,B^lv)B^lv$$ for $w\in V$.

Observe that the above operator fixes the first $n-1$ vectors of the basis $\{B^lx^0, B^lx, \ldots, B^lx^{n-1} \}$ of $V$, that is, for $k =0,\ldots, n-2$, it maps
$$B^lx^k\mapsto B^lx^k,$$ 
but it maps $$B^lx^{n-1}\mapsto B^lx^{n-1}+B^lv$$ (recall that $\Om(v,x^{n-1})=1$). 

Observe that the above operator is the transvection $B^lCB^{-l}$. We denote this transvection as $C_{B^lv}$ and note that it is defined as 
$$C_{B^lv}(B^lx^k) = \begin{cases}
	B^lx^k,          & \text{if } k=0,1,\ldots n-2,\\
	B^lx^{n-1}+B^lv, & \text{if } k =n-1.
\end{cases}$$

Recall the $4$-dimensional subspace $W$ of $V$ spanned by the vectors $B^{-1}v, v, B^2v, B^3v$ and the vector space isomorphism $$\phi: V_0\to W\subset V,  \mbox{ defined by sending } B^l_0v_0\mapsto B^lv, \text{ for }l=-1,0,2,3$$ and $\Om_0$, the symplectic form preserved by $\Gamma(f_0,g_0)=\langle A_0,B_0\rangle$ on the vector space $V_0$, from the paragraph preceding to Lemma \ref{lc in x^3 and x^{n-1}}. With these notations, we get the following lemma.

\begin{lem}\label{symplectic iso}
	The symplectic spaces $(V_0,\Om_0)$ and $(W,\Om\vert_W)$ are symplectically isomorphic via the linear map $\phi$ mentioned above. In particular, $\Om\vert_W$ is non-degenerate. Write the orthogonal decomposition of $V$ with respect to $\Om$ as $V = W\oplus W^\perp$. The group generated by the transvections $B^{-1}CB, C, B^{3}CB^{-3}$ acts trivially on $W^\perp$, and this group is $\phi$-conjugate of the group generated by the transvections $B_0^{-1}C_0B_0,C_0,$ $B_0^3C_0B_0^{-3}$.
\end{lem}

\begin{proof}
	To prove the linear map $\phi$ to be a symplectic isomorphism, the only condition that needs to be verified is that it preserves the symplectic form, that is, $\Om_0(B_0^kv_0,B_0^lv_0)$ and $\Om(B^kv,B^lv)$ coincide for $k,l\in \{-1,0,2,3\}$. By Remark \ref{form calculation}, we only need to verify that $\Om_0(v_0, B_0^lv_0)$ and $\Om(v,B^lv)$ coincide for $1\le l\le 4$. By the same Remark \ref{form calculation}, we know that $\Om_0(v_0,B_0^lv_0)$ is the coefficient of $x^3$ in the vector $B_0^lv_0 =x^{l-1}(f_0(x)-g_0(x))$ when it is expressed as a linear combination of $1,x,x^2,x^3$. Similarly, writing the vector $B^lv =x^{l-1}(f(x)-g(x))$ in a linear combination of $1,x,\cdots, x^{n-1}$ gives $\Om(v,B^lv)$, which is the coefficient of $x^{n-1}$ in the vector $B^lv$. According to Lemma \ref{lc in x^3 and x^{n-1}}, the coefficient of $x^{3}$ in $x^{l-1}(f_0(x)-g_0(x))$ matches that of $x^{n-1}$ in $x^{l-1}(f(x)-g(x))$. Hence the linear map $\phi$ preserves the symplectic form.
	
	Let $w\in W^\perp$. We wish to show that $C_{B^lv}(w) = w$ for $l = -1,0,2,3$. We have $C_{B^lv}(w) = w - \Om(w, B^lv)B^lv$, for all $l = -1,0,2,3$. But, $\Om(w,B^lv) = 0$. Hence half of the second part of the lemma follows. 
	
	Using the fact  that $\phi$ is a symplectic isomorphism, we obtain that
	\begin{align}\label{symp iso of transvections}
		\phi\circ C_{B_0^lv_0}\circ \phi^{-1}(w) &= \phi\circ C_{B_0^lv_0}(\phi^{-1}(w))\nonumber \\ 
        &=\phi\left(\phi^{-1}(w)-\Omega_0(\phi^{-1}(w),B_0^lv_0)B_0^lv_0\right)\nonumber\\
        &=w-\Omega_0(\phi^{-1}(w),B_0^lv_0)\phi(B_0^lv_0)\nonumber\\
        &=w-\Omega(w,\phi(B_0^lv_0))\phi(B_0^lv_0)\nonumber\\
        &=C_{\phi(B_0^lv_0)}(w)\nonumber\\
        &=C_{B^lv}(w)\nonumber
	\end{align} 
    for all $w\in W$ and for $l = -1,0,2,3$. 
     It shows that $\phi\circ C_{B_0^lv_0}\circ \phi^{-1} =C_{B^lv}\vert_W$, where $C_{B^lv}\vert_W$ denotes the restriction of the transvection $C_{B^lv}$ on $W$. This completes the proof.
\end{proof}

The preparatory work we performed for a particular pair $((x-1)^4, (x^2+x+1)(x^2-x+1))$ needs to be done for all pairs listed in the Table \ref{Table 1}. For pairs whose associated transvections (cf. Table \ref{Table 1}) in $\Gamma(f_0,g_0)$ are conjugates of $C_0$ by powers of $B_0$, we set $V = \mathbb{Q}[x]/\langle g(x)\rangle$ and $V_0 = \mathbb{Q}[x]/\langle g_0(x)\rangle$. For pairs whose transvections (cf. Table \ref{Table 1}) in $\Gamma(f_0,g_0)$ are conjugates of $C_0$ by powers of $A_0$, we instead set $V = \mathbb{Q}[x]/\langle f(x)\rangle$ and $V_0 = \mathbb{Q}[x]/\langle f_0(x)\rangle$, which dictates that the operator $A$ acts on $V$ by multiplication by $x$. In this case, in Lemma \ref{lc in x^3 and x^{n-1}}, the division is to be performed by $f_0(x)P_m(x^p)$ and $f_0$ in place of $g_0(x)Q_m(x^p)$ and $g_0$, respectively.

To extend the above method to all the pairs $(f_0,g_0)$ of Table \ref{Table 1}, we must find a basis for $V_0$ and identify a suitable $4$-dimensional subspace $W$ of $V$. The subspace $W$ should be spanned by the vectors obtained by replacing $A_0, B_0$, and $v_0$ with $A, B$, and $v$, respectively, in our basis vectors for $V_0$; this allows us to define the natural linear isomorphism $\phi$ of $V_0$ with $W$. To get a lemma similar to Lemma \ref{lc in x^3 and x^{n-1}} for the other examples of Table \ref{Table 1}, we must specify a maximal integer value $q$ (say) for $j$ to get the statement of Lemma \ref{lc in x^3 and x^{n-1}} (where $q$ is $3$) hold for all $j\le q$, and obtain the minimal integer $p$ for which Lemma \ref{lc in x^3 and x^{n-1}} holds and Lemma \ref{symplectic iso} follows. 

Consequently, we fix three transvections $C_1, C_2=C$, and $C_3$ in $\Gamma(f,g)$, which respectively correspond to the given transvections of $\Gamma(f_0,g_0)$ in Table \ref{Table 1} with the subscript ``$0$" omitted. Thus, for each $(f_0,g_0)$ pair, the subgroup generated by $C_1,C_2,C_3$ of $\Gamma(f,g)$ is $\phi$-conjugate to the corresponding subgroup generated by the three transvections of $\Gamma(f_0,g_0)$ listed in Table \ref{Table 1}. Moreover, the former group acts trivially on $W^\perp$. We list all the required information in the following table. To avoid redundancy, the polynomial pairs are referenced using their respective indices from Table \ref{Table 1}.

\begin{table}[htbp]
	\centering
	\setlength{\tabcolsep}{6pt} 
	\renewcommand{\arraystretch}{1.4} 
	\caption{A basis of $V_0$, the maximal integer value $q$ for $j$ and the minimal integer value of $p$ for the polynomial pairs listed in Table \ref{Table 1}.}
	\label{Table 2}
	\begin{tabular}{ccccc}
		\toprule
		No. & Basis of $V_0$  & $q$ & Minimal $p$ \\ 
		\midrule
		1 & $\{B_0^{-3}v_0, v_0, B_0v_0, B_0^3v_0\}$ & $5$ & $7$ \\
		2 & $\{B_0^{-1}v_0, v_0, B_0^2v_0, B_0^3v_0\}$  & $3$ & $5$\\
		3 & $\{B_0^{-5}v_0, v_0, B_0^2v_0, B_0^3v_0\}$ & $7$ & $9$ \\
		4 & $\{A_0^{-4}v_0, v_0, A_0v_0, A_0^4v_0\}$  & $7$ & $9$ \\
		5 & $\{A_0^{-4}v_0, v_0, A_0v_0, A_0^4v_0\}$  & $7$ & $9$ \\
		6 & $\{B_0^{-3}v_0, v_0, B_0v_0, B_0^3v_0\}$ & $5$ & $7$ \\
		7 & $\{A_0^{-4}v_0, v_0, A_0v_0, A_0^4v_0\}$  & $7$ & $9$ \\
		8 & $\{B_0^{-3}v_0, v_0, B_0v_0, B_0^3v_0\}$  & $5$ & $7$ \\
		\bottomrule
	\end{tabular}
\end{table}

Thus, we conclude that $p=9$ is the minimal integer for which the Theorem \ref{main thm} holds for all the pairs of Table \ref{Table 1}. Hence, from now on $f(x) = f_0(x)P_m(x^9)$ and $g(x) = g_0(x)Q_m(x^9)$.

We are now set to prove Theorem \ref{main thm} using the method of the proof of \cite[Theorem 1.2]{S-V14}.

\subsection{Proof of Theorem \ref{main thm}}

It has been shown in \cite{SS15, SS17} that the subgroup generated by the three transvections, appearing in Table \ref{Table 1}, of the hypergeometric group $\Gamma(f_0,g_0)$ contains non-trivial unipotent elements corresponding to both the highest and second highest roots of the symplectic group $\Sp(4)$. It follows from the discussion following Lemma \ref{symplectic iso} that the subgroup generated by $C_1,C_2, C_3$, which acts trivially on $W^\perp$, of $\Gamma(f,g)$ contains copies of those unipotent elements and these copies are unipotent in $\Sp(n)$. Since the symplectic form $\Omega\vert_W$ is non-degenerate, there exist a basis $\{\e_1,\e_2,\e_2^*,\e_1^*\}$ of $W$ such that
$$\Om(\e_1,\e_2) = 0,\quad \Om(\e_1^*,\e_2^*)=0,\quad \Om(\e_i,\e_j^*) = \delta_{ij}, \text{ for }i,j =1,2$$
where 
$$\delta_{ij} = \begin{cases}
	1, &         \text{if } i=j,\\
	0, &         \text{otherwise }.
\end{cases}$$

As observed in Lemma \ref{symplectic iso}, we can write the vector space $V$ as the orthogonal direct sum, $V = W\oplus W^\perp$. We now extend the basis $\{\e_1,\e_2,\e_2^*,\e_1^*\}$ of $W$ to a basis of $V$ by choosing a basis for $W^\perp$ as $\{\e_3,\ldots,\e_{\frac{n}{2}}, \e_{\frac{n}{2}}^*,\ldots, \e_3^*\}$ in such a way that 
$$ \Om(\e_i,\e_j) = 0,\quad \Om(\e_i^*,\e_j^*) = 0,\quad \Om(\e_i,\e_j^*)=\delta_{ij}, \text{ for }i,j = 1,\ldots, \frac{n}{2}.$$
Now, take $\{\e_1,\e_2,\e_3,\ldots, \e_{\frac{n}{2}}, \e_{\frac{n}{2}}^*,\ldots, \e_3^*,\e_2^*,\e_1^*\}$ as an ordered basis of $V$. We observe that, with respect to this basis of $V$, the subgroup generated by the three transvections $C_1, C_2, C_3$ in $\Gamma(f,g)$ contains elements of the form 

$$ \begin{pmatrix}
	1 & 0 & 0 & 0 & *\\
	0 & 1 & 0 & 0 & 0 \\
	0 & 0 & I_{n-4} & 0 & 0\\
	0 & 0 & 0 & 1 & 0 \\
	0 & 0 & 0 & 0 & 1
\end{pmatrix}, \quad
\begin{pmatrix}
	1 & 0 & 0 & * & 0\\
	0 & 1 & 0 & 0 & * \\
	0 & 0 & I_{n-4} & 0 & 0\\
	0 & 0 & 0 & 1 & 0 \\
	0 & 0 & 0 & 0 & 1
\end{pmatrix}.$$

Also, with respect to the above basis of $V$, the group of diagonal matrices in $\Sp(n)$ forms a maximal torus $T$ and the group of upper triangular matrices in $\Sp(n)$ forms a Borel subgroup $B$ of $\Sp(n)$. Let $\Phi$ be the root system of the symplectic group $\Sp(n)$ associated to the maximal torus $T$. Let $\Phi^+$ be the set of positive roots of $\Phi$ that determines the Borel subgroup $B$. Now, it is immediate that the above two non-trivial unipotent elements, found in $\Gamma(f,g)$, respectively, correspond to the highest and second highest roots in $\Phi^+$. 

Since $\Gamma(f,g)$ is Zariski dense in $\Sp(n)$ and it contains non-trivial unipotent elements corresponding to the highest and second highest roots of $\Sp(n)$, it follows from Venkataramana's arithmeticity criterion \cite[Theorem 3.5]{TNV87} that $\Gamma(f,g)$ is arithmetic in $\Sp(n)$.  \qed

\section{Remarks on generalization} \label{section 5}

A natural follow-up question arises: \textit{can we use the same technique to produce more infinite families of arithmetic subgroups of the integral symplectic group}? The answer is `\textit{Yes}'. Instead of starting with a pair of degree four polynomials, we can start with one of degree six. We list pairs of degree six polynomials in Table \ref{Table 3}, whose arithmeticity is known from \cite{B-D-S-S21}, for which one can prove the following:

``{\textit {Let $m \in\N$, and let $(f_0, g_0)$ be a pair of polynomials listed in Table \ref{Table 3}. Suppose $P_m$ and $Q_m$ are monic, self-reciprocal polynomials of degree $2m$ with integer coefficients, satisfying $P_m(0) = Q_m(0) = 1$. Let $f(x) = f_0(x)P_m(x^6)$ and $g(x) = g_0(x)Q_m(x^6)$. If $f$ and $g$ are coprime, then the group $\Gamma(f,g)$ is an arithmetic subgroup of the symplectic group $\operatorname{Sp}({12m+6})$.}}"

As noted in \cite{B-D-S-S21}, an element $\gamma_0\in \Gamma(f_0,g_0)$ provides three transvections, $\gamma_0^{-1}C_0\gamma_0, C_0,$ $\gamma_0 C_0\gamma_0^{-1}$, which contains unipotent elements corresponding to the highest and second highest roots of the corresponding symplectic group. In Table \ref{Table 3}, we list only the simplest cases, for which, the hypotheses of the above paragraph holds: here we consider the cases where $\gamma_0 = B_0^2$. Let us denote the parameters of $f_0$ and $g_0$ by $\alpha_0$ and $\beta_0$ respectively. The other cases, which involve the conjugations either by powers of $A_0$ or by powers of $B_0$, can also be dealt with in a similar fashion and $x^6$ appearing in $f(x) = f_0(x)P(x^6)$ will vary depending on the particular pair, provided we are able to find the particular basis of $V_0$, similar to the degree four cases. 

\begin{table}[htbp]
	\centering
	\setlength{\tabcolsep}{6pt} 
	\renewcommand{\arraystretch}{1.4} 
	\caption{A list of pairs of parameters of degree six, for which, the idea of the proof of Theorem \ref{main thm} works.}
	\label{Table 3}
	\begin{tabular}{cccccc}
		\toprule
		No. & $\alpha_0$  &$\beta_0$ & No. & $\alpha_0$  &$\beta_0$  \\ 
		\midrule
		1 & $\left(0,0,0,0,\frac{1}{2}, \frac{1}{2}\right)$ & $\left(\frac{1}{3}, \frac{2}{3},\frac{1}{4}, \frac{3}{4},\frac{1}{4}, \frac{3}{4}\right)$  &
		2 & $\left(0,0,0,0,\frac{1}{2}, \frac{1}{2}\right)$ & $\left(\frac{1}{3}, \frac{2}{3},\frac{1}{8}, \frac{3}{8},\frac{5}{4}, \frac{7}{8}\right)$ \\
		3 & $\left(0,0,0,0,\frac{1}{2}, \frac{1}{2}\right)$ & $\left(\frac{1}{3}, \frac{2}{3},\frac{1}{12}, \frac{5}{12},\frac{7}{12}, \frac{11}{12}\right)$ &
		4 & $\left(0,0,0,0,\frac{1}{2}, \frac{1}{2}\right)$ & $\left(\frac{1}{7}, \frac{2}{7},\frac{3}{7}, \frac{4}{7},\frac{5}{7}, \frac{6}{7}\right)$ \\
		5 & $\left(0,0,0,0,\frac{1}{3}, \frac{2}{3}\right)$ & $\left(\frac{1}{4}, \frac{3}{4},\frac{1}{8}, \frac{3}{8},\frac{5}{8}, \frac{7}{8}\right)$ &
		6 & $\left(0,0,0,0,\frac{1}{3}, \frac{2}{3}\right)$ & $\left(\frac{1}{5}, \frac{2}{5}, \frac{3}{5},\frac{4}{5}, \frac{1}{6},\frac{5}{6}\right)$ \\
		7 & $\left(0,0,0,0,\frac{1}{4}, \frac{3}{4}\right)$ & $\left(\frac{1}{3}, \frac{2}{3},\frac{1}{6}, \frac{1}{6},\frac{5}{6}, \frac{5}{6}\right)$ &
		8 & $\left(0,0,\frac{1}{2},\frac{1}{2}, \frac{1}{3}, \frac{2}{3}\right)$ & $\left(\frac{1}{4}, \frac{3}{4},\frac{1}{6}, \frac{1}{6},\frac{5}{6}, \frac{5}{6}\right)$ \\
		9 & $\left(0,0,\frac{1}{2},\frac{1}{2}, \frac{1}{3}, \frac{2}{3}\right)$ & $\left(\frac{1}{6}, \frac{5}{6},\frac{1}{10}, \frac{3}{10},\frac{7}{10}, \frac{9}{10}\right)$ & 
		10 & $\left(0,0,\frac{1}{2},\frac{1}{2}, \frac{1}{4}, \frac{3}{4}\right)$ & $\left(\frac{1}{3}, \frac{1}{3},\frac{1}{3}, \frac{2}{3},\frac{2}{3}, \frac{2}{3}\right)$\\
		11 & $\left(0,0,\frac{1}{5},\frac{2}{5}, \frac{3}{5}, \frac{4}{5}\right)$ & $\left(\frac{1}{3}, \frac{1}{3},\frac{2}{3}, \frac{2}{3},\frac{1}{4}, \frac{3}{4}\right)$ &
		12 & $\left(0,0,\frac{1}{8},\frac{3}{8}, \frac{5}{8}, \frac{7}{8}\right)$ & $\left(\frac{1}{3}, \frac{1}{3},\frac{2}{3}, \frac{2}{3},\frac{1}{6}, \frac{5}{6}\right)$\\
		13 & $\left(0,0,\frac{1}{8},\frac{3}{8}, \frac{5}{8}, \frac{7}{8}\right)$ & $\left(\frac{1}{3}, \frac{2}{3},\frac{1}{4}, \frac{3}{4},\frac{1}{4}, \frac{3}{4}\right)$ &
		14 & $\left(0,0,\frac{1}{8},\frac{3}{8}, \frac{5}{8}, \frac{7}{8}\right)$ & $\left(\frac{1}{4}, \frac{3}{4},\frac{1}{5}, \frac{2}{5},\frac{3}{5}, \frac{4}{5}\right)$\\
		15 & $\left(0,0,\frac{1}{10},\frac{3}{10}, \frac{7}{10}, \frac{9}{10}\right)$ & $\left(\frac{1}{3}, \frac{2}{3},\frac{1}{4}, \frac{3}{4},\frac{1}{6}, \frac{5}{6}\right)$ &
		16 & $\left(0,0,\frac{1}{10},\frac{3}{10}, \frac{7}{10}, \frac{9}{10}\right)$ & $\left(\frac{1}{4}, \frac{1}{4},\frac{1}{4}, \frac{3}{4},\frac{3}{4}, \frac{3}{4}\right)$\\
		17 & $\left(0,0,\frac{1}{12},\frac{5}{12}, \frac{7}{12}, \frac{11}{12}\right)$ & $\left(\frac{1}{3}, \frac{2}{3},\frac{1}{8}, \frac{3}{8},\frac{5}{8}, \frac{7}{8}\right)$ &
		18 & $\left(0,0,\frac{1}{12},\frac{5}{12}, \frac{7}{12}, \frac{11}{12}\right)$ & $\left(\frac{1}{4}, \frac{3}{4},\frac{1}{5}, \frac{2}{5},\frac{3}{5}, \frac{4}{5}\right)$\\
		\bottomrule
	\end{tabular}
\end{table}

Another question that arises is: \textit{can we use a word $\gamma_0$ in both $A_0$ and $B_0$ to conjugate $C_0$, instead of a word which is just a power of $A_0$ or $B_0$}? This approach does not seem feasible due to the restriction of the action of the operator $\gamma_0$ on $V$, which cannot be realised as a multiplication by $x$. Take, for instance, the simplest case where $\gamma_0 = B_0A_0$ (cf. \cite[S.No. 117]{B-D-S-S21}). Even if we somehow manage to find a desired basis for $V_0$, which is also not obvious, but $\{\gamma^{-1}v,v,\gamma v\}$ need not be linearly independent (see \cite[Remark 2]{B-D-S-S21}) and hence cannot be a part of the desired basis of the subspace $W$, and a result of this we cannot match the leading coefficients like what we did in Lemma \ref{lc in x^3 and x^{n-1}}. 

We now have a general question:

\begin{que}
	Does the arithmeticity of  $\Gamma(f_0,g_0)$ imply the arithmeticity of $\Gamma(f,g)$, where $f$ and $g$ are multiples of $f_0$ and $g_0$, respectively, like in Theorem \ref{main thm}?
\end{que}


\subsection*{Acknowledgements}

The author would like to express his gratitude to his PhD advisor, Sandip Singh, for introducing him to the theory of hypergeometric groups and for his unwavering support, insightful remarks, and helpful suggestions and discussions. The author also acknowledges the PhD fellowship provided by IIT Bombay.

\vspace{1cm}
 
\end{document}